\documentclass[11pt,reqno]{amsart}

\usepackage{amscd,amssymb,amsmath,amsthm}
\usepackage{graphicx}
\usepackage{color}
\usepackage{cite}
\usepackage{tikz}
\newtheorem{thm}{Theorem}
\newtheorem{defn}{Definition}
\newtheorem{lemma}{Lemma}
\newtheorem{pro}{Proposition}
\newtheorem{rk}{Remark}

\numberwithin{equation}{section} 

\makeatletter
\newcommand*{\rom}[1]{\expandafter\@slowromancap\romannumeral #1@}
\makeatother

\begin{document}
	\title [Trajectories of an operator]
	{Dynamical systems of an infinite-dimensional non-linear operator}
	\author {U.R. Olimov, U.A. Rozikov}

\address{U.\ R. Olimov\\ V.I.Romanovskiy Institute of Mathematics,  9, Universitet str., 100174, Tashkent, Uzbekistan.}
\email {umrbek.olimov.92@mail.ru}

\address{ U.Rozikov$^{a,b}$\begin{itemize}
		\item[$^a$] V.I.Romanovskiy Institute of Mathematics,  9, Universitet str., 100174, Tashkent, Uzbekistan;
		\item[$^b$] National University of Uzbekistan,  4, Universitet str., 100174, Tashkent, Uzbekistan.
\end{itemize}}
\email{rozikovu@yandex.ru}

\	\begin{abstract} We investigate discrete-time dynamical systems generated by an infinite-dimensional non-linear operator that maps the Banach space $l_1$ to itself. It is demonstrated that this operator possesses up to seven fixed points. By leveraging the specific form of our operator, we illustrate that analyzing the operator can be simplified to a two-dimensional approach. Subsequently, we provide a detailed description of all fixed points, invariant sets for the two-dimensional operator and determine the set of limit points for its trajectories. These results are then applied to find the set of limit points for trajectories generated by the infinite-dimensional operator.
\end{abstract}
\maketitle



\section{Introduction}
Our goal is to study dynamical systems generated by an infinite-dimensional non-linear operator that maps the Banach space $l_1$ to itself. Such operators appear in various areas of mathematics and physics. Examples include the Navier-Stokes equations in fluid dynamics, the nonlinear Schr\"odinger equation in quantum mechanics (representing physical observables such as position, momentum, and angular momentum), in variational problems (where one seeks to minimize or maximize a functional defined on an infinite-dimensional function space). Moreover, infinite-dimensional operators are central to the study of functional analysis, which deals with spaces of functions and their properties. Understanding the properties of these operators is essential for analyzing complex systems and phenomena in these fields.

The infinite-dimensional operator which we consider in this paper appears when investigating (gradient) Gibbs measures for models featuring an infinite set of spin values. For a detailed examination of gradient Gibbs measures associated with gradient potentials on Cayley trees, we refer to \cite{KMR, RoH, RoJSP, 9,  a, b, c} and the sources cited therein. 

The paper is structured as follows: In Section 2, we provide essential definitions, outline our problem formulation, and review relevant existing results. Section 3 focuses on a comprehensive analysis of discrete-time dynamical systems generated by an infinite-dimensional nonlinear operator that maps the Banach space $l_1$ to itself. We demonstrate that this operator has up to seven fixed points and reveal that its analysis can be streamlined into a two-dimensional framework. By leveraging our findings on the two-dimensional operator, we determine the set of limit points for trajectories produced by the infinite-dimensional operator.

\section{Preliminaries and the main problem}
Let us give some basic definitions and facts (see \cite[Chapter 1]{Romb}).
 Consider an operator $F$ defined on a topological space $X$.

For $A\subset X$ denote $F(A)=\{F(x): x\in A\}$.

If $F(A)\subset A$, then $A$ is called an invariant set under function $F$.

If $F(x)=x$ then the point $x\in X$ is called a fixed point for $F$.


Denote by $F^n(x)$ the $n$ times iterations of $F$ to $x\in X$.

For given topological space $X$, $x^{(0)}\in X$ and $F:X\to X$ the discrete-time dynamical system is defined as
\begin{equation}\label{0a2}
	x^{(0)},\ \ x^{(1)}= F(x^{(0)}), \ \ x^{(2)}=F^{2}(x^{(0)}),\ \  x^{(3)}= F^{3}(x^{(0)}),\dots
\end{equation}

{\bf The main problem:} for a given dynamical system is to
describe the limit points of $\{x^{(n)}\}_{n=0}^\infty$ for
arbitrary given $x^{(0)}$.

\begin{defn}\label{fp}
	A point $x\in X$ is called a fixed point for $F:X\to X$ if $F(x)=x$. The point $x$
	is a periodic point of period $p$ if $F^p(x) = x$.
\end{defn}

\begin{defn} A fixed point $x^*$ of $F$ 
	
	\begin{itemize}
		\item is attracting or stable if there
		is an open set (neighborhood) $U$ containing $x^*$ such that for all $x\in U$, the sequence $F^n(x)$
		with initial point $x$ converges to $x^*$.
		\item is repelling if there is an open set (neighborhood) $U$ containing
		$x^*$ such that for any $x\in U$, $x\ne x^*$, there is some $k\geq 1$ such that $F^k(x)\notin U$.
	\end{itemize}
\end{defn}
Consider a dynamical system where $X$ is a subset of $\mathbb R$
and $f$ is a continuous function on $X$.
Assume that $f$ is also continuously differentiable on $X$.

\begin{defn}
	A fixed point $x^*\in X$ is called hyperbolic if $|f'(x^*)|\ne 1$.
\end{defn}
The following theorem is well-known

\begin{thm}\label{tar}  Let $X \subset \mathbb R$ and $f$ be continuously
	differentiable on $X$. Let $x^*\in X$ be a hyperbolic fixed
	point of $f$ then
	\begin{itemize}
		\item[1)] If $|f'(x^*)|< 1$, then $x^*$ is attracting.
		\item[2)] If $|f'(x^*)|> 1$, then $x^*$ is repelling.
	\end{itemize}
\end{thm}
 Denote
$$l^{1}_{+}=\left\{x=(x_{1},x_{2},\dots,x_{n},\dots) \, : \, x_{i}>0, \, \|x\|=\sum_{j=1}^{\infty}x_j<\infty\right\}.$$

In this paper we study discrete-time, infinite-dimensional dynamical systems (IDDS) generated by operator $F$ defined on $ l^{1}_{+}$ as
\begin{equation}\label{e4o}
F:	\begin{array}{lll}
		x'_{2n-1}=\lambda_{2n-1}\cdot \left(\dfrac{1+\sum_{j=1}^{\infty}x_{2j-1}}{1+\theta+\|x\|}\right)^2,\\[2mm]
		x'_{2n}=\lambda_{2n}\cdot \left(\dfrac{1+\sum_{j=1}^{\infty}x_{2j}}{1+\theta+\|x\|}\right)^2,
	\end{array}
\end{equation}
where $n=1,2,\dots$, $\theta>0$ and $\lambda=(\lambda_1, \lambda_2, \dots)\in l^{1}_{+}$.

\begin{rk}
We note that this infinite dimensional non-linear operator is related to Gibbs measures of a physical system called a hard core model with a countable set of spin values on the Cayley tree (see \cite{OR}, \cite{KMR} and references therein).
\end{rk}

\begin{lemma} \cite{O} If $\lambda=(\lambda_1, \lambda_2, \dots)\in l^{1}_{+}$ then $F$ mapps $l^{1}_{+}$ to itself.
\end{lemma}

Define two-dimensional operator
 $W: z=(x,y)\in \mathbb R^2_+\to z'=(x',y')=W(z)\in \mathbb R^2_+$ by
\begin{equation}\label{op}
	W: \ \ \begin{array}{ll}
		x'=L_1\cdot \left(\dfrac{1+x}{1+\theta+x+y}\right)^2\\[2mm]
		y'=L_2\cdot\left(\dfrac{1+y}{1+\theta+x+y}\right)^2,
	\end{array}
\end{equation}
where $\theta>0$ and $L_i>0$ are parameters.

\begin{lemma}\label{Le} \cite{O} The IDDS generated by the operator $F$ is fully represented by the two-dimensional DS generated by the operator $W$.
\end{lemma}

\begin{rk}\label{ob}
Thus, by Lemma \ref{Le} we reduced investigation of trajectories $(x^{(m)})_{m=0}^\infty$, for all $x^{(0)}\in l^1_+$ to the study trajectories generated by the two-dimensional non-linear operator $W$ given in (\ref{op}). Namely,
for an initial point $x^{(0)}=(x_1^{(0)}, x_2^{(0)}, \dots)\in  l^{1}_{+}$, introducing
$$M^{(0)}_1=\sum_{j=1}^{\infty}x^{(0)}_{2j-1}, \ \ M^{(0)}_2=\sum_{j=1}^{\infty}x^{(0)}_{2j}$$
we note that the relation between limit points of operators $W$ and $F$ is given as:

if $\lim_{m\to \infty}W^m(M_1^{(0)}, M_2^{(0)})=(a,b)$ then
\begin{equation}\label{xab}\lim_{m\to \infty}F^m(x^{(0)})=\left({a\over L_1}\lambda_1, {b\over L_2}\lambda_2, {a\over L_1}\lambda_3, {b\over L_2}\lambda_4, \dots\right).
\end{equation} 	
\end{rk}

Therefore, it suffices to study DS of $W$. For simplicity we assume $L_1=L_2=L$.

\section{The case $L_1=L_2=L$}
Consider operator (\ref{op}) for the case $L_1=L_2=L$.
\subsection{Fixed points} In this subsection we review results of \cite{O} about fixed points of operator $F$. These are needed to describe limit points of the dynamical system.

To explicitly give fixed points we introduce the following notations:
$$N=\sqrt[3]{(45-9\theta+L)\sqrt{L}+\sqrt{216+648\theta+648\theta^2+216\theta^3+(1917-1026\theta -27\theta^2 +108L)L}}.$$
$$u_{1}=\frac{\sqrt{L}}{6}-\frac{6+6\theta-L}{6N}+\frac{1}{6}N,$$
$$u_{2}=\frac{\sqrt{L}}{6}+\frac{(1+i\sqrt{3})(6+6\theta-L)}{12N}-\frac{1}{12}(1-i\sqrt{3})N,$$
$$u_{3}=\frac{\sqrt{L}}{6}+\frac{(1-i\sqrt{3})(6+6\theta-L)}{12N}-\frac{1}{12}(1+i\sqrt{3})N,$$
Denote
\begin{equation}\label{umr} x_1=u_1^2,\ \ x_2=u_2^2, \ \ x_3=u_3^2.\end{equation}
 $$A_{1,0}=\left\{(\theta, L):L\leq\frac{(\theta+3)^2}{4},0<\theta\leq5\right\}
 \bigcup\left\{(\theta, L):L<4(\theta-1),\theta>5\right\},$$
$$A_{1,2}=\left\{(\theta, L):L>\frac{(\theta+3)^2}{4},0<\theta\leq17\right\}\bigcup\left\{(\theta, L):\frac{(\theta+3)^2}{4}\leq L<L_{1}, 17<\theta\leq9+8\sqrt{2}\right\}$$ $$\bigcup\left\{(\theta, L):L>L_{2},\theta>17\right\}\bigcup\left\{(\theta, L):L=4(\theta-1),\theta>5\right\},$$
$$A_{2,2}=\left\{(\theta, L):L=L_{2},\theta>17\right\}\bigcup\left\{(\theta, L):L=L_{1},17<\theta<9+8\sqrt{2}\right\},$$
$$A_{1,4}=\left\{(\theta, L):4(\theta-1)<L<\frac{(\theta+3)^2}{4},5<\theta\leq9+8\sqrt{2}\right\}$$ $$\bigcup\left\{(\theta, L):4(\theta-1)<L<L_{1},\theta>9+8\sqrt{2}\right\}$$
$$A_{3,2}=\left\{(\theta, L):L_{1}<L<L_{2},17<\theta\leq9+8\sqrt{2}\right\}$$ $$\bigcup\left\{(\theta, L):\frac{(\theta+3)^2}{4}<L<L_{2},\theta>9+8\sqrt{2}\right\},$$
$$A_{2,4}=\left\{(\theta, L):L=L_{1},\theta>9+8\sqrt{2}\right\},$$
$$A_{3,4}=\left\{(\theta, L):L_{1}<L<\frac{(\theta+3)^2}{4},\theta>9+8\sqrt{2}\right\}.$$

Denoting $A_{i,j}$ above means that the number of fixed points on $M_{0}=\{(x,y)\in\mathbb R_+^2: x=y\}$ is $i$ and the number of fixed points outside $M_{0}$ is $j$.

\begin{lemma}\label{lfp} \cite{O} The following assertions hold
\begin{equation}
	{\rm Fix}(W)= \left\{\begin{array}{lllllll}
		\{p^*_{1}\}, \ \ \mbox{if} \ \  (\theta, L)\in A_{1,0}\\[2mm]	
	\{p^*_{1}, p_{1}, p_{2}\}, \ \ \mbox{if} \ \
	(\theta, L)\in A_{1,2}\\[2mm]
	\{p^*_{1}, p^*_{2}, p_{1}, p_{2}\}, \ \ \mbox{if} \ \
 (\theta, L)\in A_{2,2}\\[2mm]
	\{p^*_{1}, p_{1}, p_{2}, p_{3}, p_{4}\}, \ \ \mbox{if} \ \
	(\theta, L)\in A_{1,4}\\	
	\{p^*_{1}, p^*_{2}, p^*_{3}, p_{1}, p_{2}\}, \ \ \mbox{if} \ \ (\theta, L)\in A_{3,2}\\[2mm]
	\{p^*_{1}, p^*_{2}, p_{1}, p_{2}, p_{3}, p_{4}\}, \ \ \mbox{if} \ \ (\theta, L)\in A_{2,4}\\[2mm]
	\{p^*_{1}, p^*_{2}, p^*_{3}, p_{1}, p_{2}, p_{3}, p_{4}\}, \ \ \mbox{if} \ \ (\theta, L)\in A_{3,4}.\end{array}
\right.
\end{equation}
where
$$p^*_i=(x_i, x_i), \, i=1,2,3 \ \ (\mbox{see} \ \ (\ref{umr})).$$
$$p_{1}=(t_{1},t_{2}), p_{2}=(t_{2},t_{1}), p_{3}=(t_{3},t_{4}), p_{4}=(t_{4},t_{3}),$$
with

%
%
%
%
%
%
%
%
%
%
%
%
 $$t_{1}={1\over 16}\left(\sqrt{L}+\sqrt{L-4(\theta-1)}+\sqrt{2L-4\theta-12+2\sqrt{L^2-4(\theta-1)L}}\right)^2,$$
$$t_{2}={1\over 16}\left(\sqrt{L}+\sqrt{L-4(\theta-1)}-\sqrt{2L-4\theta-12+2\sqrt{L^2-4(\theta-1)L}}\right)^2,$$
$$t_{3}={1\over 16}\left(\sqrt{L}-\sqrt{L-4(\theta-1)}+\sqrt{2L-4\theta-12-2\sqrt{L^2-4(\theta-1)L}}\right)^2,$$
$$t_{4}={1\over 16}\left(\sqrt{L}-\sqrt{L-4(\theta-1)}-\sqrt{2L-4\theta-12-2\sqrt{L^2-4(\theta-1)L}}\right)^2.$$
\end{lemma}

Thus for fixed points of the infinite-dimensional operator $F$ we have the following theorem.

\begin{thm}\label{tfp} \cite{O} Fixed points of operator (\ref{e4o}) for $L_{1}=L_{2}=L$ are as follows:

$${\rm Fix}(F)=\left\{\begin{array}{ll}
	\{P_1\}, \ \ \mbox{if} \ \ (\theta,L)\in A_{1,0},\\[2mm]
	\{P_1, P_4, P_5\} \ \ \mbox{if} \ \ (\theta,L)\in A_{1,2},\\[2mm]
	\{P_1, P_2, P_4, P_5\}  \ \ \mbox{if} \ \ (\theta,L)\in A_{2,2},\\[2mm]
	\{P_1, P_4, P_5, P_6, P_7\}  \ \ \mbox{if} \ \ (\theta,L)\in A_{1,4},\\[2mm]
	\{P_1, P_2, P_3, P_4, P_5\}  \ \ \mbox{if} \ \ (\theta,L)\in A_{3,2},\\[2mm]
	\{P_1, P_2, P_4, P_5, P_6, P_7\} \ \ \mbox{if} \ \ (\theta,L)\in A_{2,4},\\[2mm]
	\{P_1, P_2, P_3, P_4, P_5, P_6, P_7\} \ \ \mbox{if} \ \ (\theta,L)\in A_{3,4};
\end{array}\right.$$
where
$$\begin{array}{ll}
	P_i={x_i\over L}(\lambda_1, \lambda_2, \lambda_3,...), \ \ i=1,2,3.\\[3mm]
		P_4={1\over L}(t_1\lambda_1, t_2\lambda_2, t_1\lambda_3,...), \ \ P_5={1\over L}(t_2\lambda_1, t_1\lambda_2, t_2\lambda_3,...),\\[3mm]
	P_6={1\over L}(t_3\lambda_1, t_4\lambda_2, t_3\lambda_3,...), \ \
	P_7={1\over L}(t_4\lambda_1, t_3\lambda_2, t_4\lambda_3,...).
\end{array}$$
\end{thm}

\subsection{Invariant sets}
Denote
$$M_{-}=\{(x,y)\in \mathbb R^2_+: x<y\},$$
$$M_0=\{(x,y)\in \mathbb R^2_+: x=y\},$$
$$M_+=\{(x,y)\in \mathbb R^2_+: x>y\}.$$
\begin{lemma}\cite{O} If $L_1=L_2=L$ then the sets $M_\epsilon$, $\epsilon=-,0,+$ are invariant with respect to operator $W$, i.e., $W(M_\epsilon)\subset M_\epsilon$.
\end{lemma}

For each initial point $v^{(0)}=(x_0, y_0)\in\mathbb R^2_+$ consider its trajectory under operator (\ref{op}):
\begin{equation}\label{tr}
	v^{(n)}=(x_n, y_n)=W^n(v^{(0)}), \ \ n\geq 1.
\end{equation}
Since $M_\epsilon$ is an invariant, it suffices to study trajectories on each such set separately.

\subsection{Dynamics on the invariant set $M_0$}

Reduce the operator $W$ defined by (\ref{op}) on the invariant set $M_0$, then we get
\begin{equation}\label{f}
x'=f(x):=L\left(\frac{1+x}{1+\theta+2x}\right)^2.
\end{equation}
Denote
$$\hat L_1=\frac{2\theta^2+76\theta-142-(2\theta-34)\sqrt{\theta^2-18\theta+17}}{16},$$

$$\hat L_2=\frac{2\theta^2+76\theta-142+(2\theta-34)\sqrt{\theta^2-18\theta+17}}{16}.$$
By \cite[Section 4.1]{OR} and Lemma \ref{lfp} we know that
\begin{itemize}
\item[1)] If $\theta\in(0, 17]$, $L>0$ or $\theta>17$, $L\notin(\hat L_1, \hat L_2)$, then $f$ (defined in (\ref{f})) has unique fixed point, denoted by $x_1$;

\item[2)] If $\theta>17$ and $L=\hat L_1$ or $L=\hat L_2$ then $f$ has two fixed points, denoted by $x_1$, $x_2$ with $x_1<x_2$;

\item[3)] If $\theta>17$, $L\in(\hat L_1, \hat L_2)$ then $f$ has three fixed points $x_i$, $i=1, 2, 3$, with $x_1<x_2<x_3$.
\end{itemize}
 If $\theta=1$ then $f(x)=\frac{L}{4}$ therefore we consider the case $\theta\ne 1$.
By Theorem \ref{tar} for fixed points of function (\ref{f}) we have
\begin{lemma}\label{f5} The types of fixed points are as follows
\begin{itemize}
\item[1)] The unique fixed point
$$x_1= \left\{\begin{array}{lll}
	attracting, \ \ \mbox{if} \ \ \theta>17, L\notin(\hat L_1,\hat L_2),\\[2mm]
saddle, \ \ \mbox{if} \ \ \theta=17, L=108,\\[2mm]
attarcting, \ \ \mbox{if} \ \ \theta=17, L\neq108 \ \ \mbox{or} \ \ \theta\in(0,1)\cup(1,17)\\[2mm]
\end{array}\right.$$

\item[2)] If $\theta>17$ and $L=\hat L_1$  (resp. $L=\hat L_2$) then the function $f$ has two fixed points $x_1<x_2$ and $x_1$ is saddle and $x_2$ is attracting (resp. $x_1$ is attracting and $x_2$ is saddle).

\item[3)] If $\theta>17,L\in(\hat L_1,\hat L_2)$ then $f$ has three fixed points with $x_1<x_2<x_3$. Moreover, $x_1$ and $x_3$ are attracting and $x_2$ is repelling.
\end{itemize}
\end{lemma}
\begin{proof} To determine type of fixed points we calculate
$|f'(x_i)|$ for each fixed point $x_i$, $i=1,2,3$.
Using the equality
\begin{equation}\label{ff}
	x_i=f(x_i)=L\left(\frac{1+x_i}{1+\theta+2x_i}\right)^2.
\end{equation}
We obtain
$$f'(x_i)=2L\frac{(1+x_i)(\theta -1)}{(1+\theta+2x_i)^3}=\frac{2(\theta -1)x_i}{(1+x_i)(1+\theta+2x_i)}.$$

I. For $\theta\in (0,1]\cup [17, +\infty])$ we denote
$$x_1^{*}=\frac{\theta-5-\sqrt{\theta^2-18\theta+17}}{4}, \ \
x_2^{*}=\frac{\theta-5+\sqrt{\theta^2-18\theta+17}}{4}.$$
Simple computations show that

a) if  $x_i<x_1^{*}$ or $x_i>x_2^{*}$ then $|f'(x_i)|<1$ ,

b) if $x_1^{*}<x_i<x_2^{*}$ then $|f'(x_i)|>1$ ,

c) If $x_i=x_1^{*}$ or $x_i=x_2^{*}$ then $|f'(x_i)|=1$.

II.  If $\theta\in (1, 17)$ then there is unique fixed point $x_1$ and $|f'(x_1)|<1$.

Consider
$$g(x)=f(x)-x.$$

\textbf{1) Unique fixed point.}

a) If $\theta>17$, $L<\hat L_1$ then
$$g(x)=f(x)-x=L\left(\frac{1+x}{1+\theta+2x}\right)^2-x$$
$$=(L-\hat L_2)\left(\frac{1+x}{1+\theta+2x}\right)^2+\hat L_2\left(\frac{1+x}{1+\theta+2x}\right)^2-x.$$
Consequently,
$g(0)>0$ and since
$$\hat L_2\left(\frac{1+x_1^*}{1+\theta+2x_1^*}\right)^2-x_1^*=0$$ we get
$g(x_1^{*})<0$. Therefore,  $x_1<x_1^{*}$ and   $|f'(x_1)|<1$.

b) If $\theta>17$, $L>\hat L_2$ then
$$g(x)=f(x)-x=L\left(\frac{1+x}{1+\theta+2x}\right)^2-x$$ $$=(L-\hat L_1)\left(\frac{1+x}{1+\theta+2x}\right)^2+\hat L_1\left(\frac{1+x}{1+\theta+2x}\right)^2-x.$$
Therefore
$g(L)<0$ and $g(x_2^{*})>0$. Consequently, $x_1>x_2^{*}$ and thus $|f'(x_1)|<1$.

c) If $\theta=17$ then  $\hat L_1=\hat L_2=108$,  $x_1^{*}=x_2^{*}=3$. In both case $L>108$ or $L<108$ we have   $|f'(x_1)|<1$. Moreover, for $L=108$ we have $x_1=3$ and  $f'(3)=1$.

d) For $\theta=1$ we have $f(x)=\frac{L}{4}$.

e) If $0<\theta<1$ then $$|f'(x_1)|=\bigg|\frac{2(\theta -1)x_1}{(1+x_1)(1+\theta+2x_1)}\bigg|=\bigg|\frac{\theta -1}{1+x_1}\bigg|\cdot\frac{2x_1}{1+\theta+2x_1}<1.$$

f) If  $1<\theta<17$ then as mentioned above we have $|f'(x_1)|<1$.

\textbf{2) Two fixed points.}
For
$$L=\hat L_2=\frac{2\theta^2+76\theta-142+(2\theta-34)\sqrt{\theta^2-18\theta+17}}{16}$$ we have
$$x_{1}=\frac{\theta-5-\sqrt{\theta^2-18\theta+17}}{4}, \ \ x_2=\frac{\theta^2-10\theta-23+(\theta-1)\sqrt{\theta^2-18\theta+17}}{32}.$$
In this case it is easy to see that $x_2^{*}<x_2$ and
$$f'(x_1)=1, \ \ |f'(x_2)|<1.$$

For
$$L=\hat L_1=\frac{2\theta^2+76\theta-142-(2\theta-34)\sqrt{\theta^2-18\theta+17}}{16}$$
we have
$$x_{1}=\frac{\theta-5+\sqrt{\theta^2-18\theta+17}}{4}, \ \
x_2=\frac{\theta^2-10\theta-23-(\theta-1)\sqrt{\theta^2-18\theta+17}}{32}.$$
Since $x_1^{*}>x_2$ we get
$$|f'(x_1)|<1, \ \ f'(x_2)=1.$$

\textbf{3) Three fixed points.}
If $\hat L_1<L<\hat L_2$ there are three fixed points: $x_1<x_2<x_3$.
We have

a) $g(0)=\frac{L}{(1+\theta)^2}>0$;

b) $$g(x)=(L-\hat L_2)\left(\frac{1+x}{1+\theta+2x}\right)^2+\hat L_2\left(\frac{1+x}{1+\theta+2x}\right)^2-x.$$
Therefore $g(x_1^{*})<0$.

c) From $$g(x)=(L-\hat L_1)\left(\frac{1+x}{1+\theta+2x}\right)^2+\hat L_1\left(\frac{1+x}{1+\theta+2x}\right)^2-x$$ we get $g(x_2^{*})>0$.

d) $g(L)=L\left(\left(\frac{1+L}{1+\theta+2L}\right)^2-1\right)<0$.

By the above mentioned inequalities we get  $$|f'(x_1)|<1, \ \ |f'(x_2)|>1, \ \ |f'(x_3)|<1.$$
 \end{proof}

\begin{lemma} The operator (\ref{op}), on the invariant set $M_0$, does not have any $p$-periodic point with $p\geq 2$.
\end{lemma}
\begin{proof}
For simplicity we set $b=\frac{L}{4}$, $a=\frac{1+\theta}{2}$
then $$f(x)=b\left(\frac{1+x}{a+x}\right)^2.$$
First consider the case $p=2$ then 2-periodic points different from fixed are solutions to
\begin{equation}\label{2p}
	\frac{f(f(x))-x}{f(x)-x}=0.
\end{equation}	
This equation is equivalent to
$$(a+b)^2x^2+(2a^3+a^2b+4ab+2b^2-b)x+(a^2+b)^2=0.$$
Since  $a> 1/2$ and $a\neq1$  we have
$$D=-b(a-1)^2(4a^3+3a^2b+6ab+4b^2-b)<0.$$
Therefore, the equation (\ref{2p}) does not have solutions.
Now using Sharkovskii's theorem (see \cite[Corollary 1.1]{Romb}) we conclude that $f(x)$ does not have any $p$-periodic point for $p\geq2$.
\end{proof}
The following theorem describes all limit points on $M_0$.
\begin{thm} The following assertions hold
\begin{itemize}
	\item[1)] If $\theta\in(0, 17]$, $L>0$ or $\theta>17$, $L\notin(\hat L_1, \hat L_2)$ then for any $x\in( 0,+\infty)$ the following equality holds
	$$\lim_{n\to\infty}f^n(x)= x_1.$$
	
	\item[2)] If $\theta>17$ and $L=\hat L_1$  (resp. $L=\hat L_2$) then
	$$\lim_{n\to\infty}f^n(x)=\left\{\begin{array}{ll}
		x_1, \ \ \mbox{if} \ \ x\in (0, x_1]\\[2mm]
		x_2, \ \ \mbox{if} \ \ x\in (x_1, +\infty)
		\end{array}\right.,$$
 $$	\left({\rm resp.} \ \ \lim_{n\to\infty}f^n(x)=\left\{\begin{array}{ll}
			x_1, \ \ \mbox{if} \ \ x\in (0, x_2)\\[2mm]
			x_2, \ \ \mbox{if} \ \ x\in [x_2, +\infty)\end{array}\right. \right)$$
	
	\item[3)] If $\theta>17,L\in(\hat L_1,\hat L_2)$ then 	$$\lim_{n\to\infty}f^n(x)=\left\{\begin{array}{lll}
		x_1, \ \ \mbox{if} \ \ x\in (0, x_2)\\[2mm]
		x_2, \ \ \mbox{if} \ \ x=x_2\\[2mm]
		x_3, \ \ \mbox{if} \ \ x\in (x_2, +\infty)
	\end{array}\right..$$
\end{itemize}
%
%
%
%
%
%
%
%
%
%
%
%
%
%
%
\end{thm}
\begin{proof} Let us prove part 3), i.e., the case when the function $f$ has three fixed points $x_i$, $i=1,2,3$. This proof is more simple for the parts 1) and 2).
	
3) For $\theta>17,L\in(\hat L_1,\hat L_2)$ we have $f'(x)>0$, i.e., $f$ is an increasing function. In this case we have three fixed points $x_1<x_2<x_3$. Moreover, by Lemma \ref{f5} we know that $x_2$ is repeller and the points $x_1, x_3$ are attractive.
We shall take arbitrary $x^{(0)}>0$ and prove that $x^{(n)} =f(x^{(n-1)})$, $n\geq 1$ converges as $n\to\infty$. Consider the following partition
$$(0,+\infty) = (0, x_1)\cup\{x_1\}\cup(x_1, x_2)\cup \{x_2\}\cup(x_2, x_3)\cup \{x_3\}\cup(x_3,+\infty).$$ Since $f$ is an increasing function, for any  $x\in (0, x_1)$ we have $x < f(x) < x_1$. From the last inequalities we get  $x<f(x) < f^2(x) < f(x_1)=x_1$ and iterating we get $f^{n-1}(x)<f^{n}(x)<x_1$, which for any $x^{(0)}\in (0, x_1)$ gives $x^{(n-1)}<x^{(n)}<x_1$, i.e., $x^{(n)}$ converges and its limit is a fixed point of $f$, since $f$ has unique fixed point $x_1$ in $(0, x_1]$ we conclude that the limit is $x_1$. For $x\in (x_1, x_2)$
we have $x_2>x >f(x)>x_1$, consequently $x^{(n)} > x^{(n+1)}$, i.e., $x^{(n)}$ converges
and its limit is again $x_1$. Similarly, one can show that if $x^{(0)}>x_2$ then $x^{(n)}\to x_3$ as $n\to \infty$.	
\end{proof}

\subsection{On the invariant set $M_{-}$}
Consider operator $W$ on the invariant set $M_-$.
\subsubsection{The set of limit points of trajectory}
For any initial vector  $v^{(0)}=(x^{(0)}, y^{(0)})\in \mathbb R^2_+$ its trajectory under operator $W$ is defined by (\ref{tr}).

Denote by $\omega(v^{(0)})$ the set of all limit points of the trajectory $\{v^{(n)}\}_{n=0}^\infty$ started at $v^{(0)}$.

Note that if the trajectory converges, then $\omega(v^{(0)})$ consists of a single point. If the trajectory does not converge, then $\omega(v^{(0)})$ is finite (with more than one element) if the trajectory is periodic; otherwise, the set of limit points is an infinite set.

In this subsection we study the set of limit points.
Note that operator (\ref{op}) is well defined at point $v^{(0)}=\ell:=(0, L)$.
Consider trajectory $\ell_n=W^n(\ell)=(l_1^{(n)}, l_2^{(n)})$.

\begin{lemma}\label{ell} Independently on values of parameters $\theta>0$ and $L>0$ the trajectory $\ell_n$ has limit:
	$$\lim_{n\to\infty}\ell_n\to \hat\ell\in {\rm Fix}(W).$$
\end{lemma}
\begin{proof} Since function $f(x,y)=L({1+x\over 1+\theta+x+y})^2$ is monotone increasing with respect to $x$ and decreasing with respect to $y$, by formula of (\ref{op}) it is clear that, independently on initial point $v^{(0)}=(x_0, y_0)\in M_-$, its trajectory
 $v^{(n)}=(x_n, y_n), \forall n\geq 1,$
 can be bounded as follows:
 $$x_0>0, y_0>0,$$
 $$0=l_1^{(0)}<x_1<l_2^{(0)}=L, \ \ \ 	0=l_1^{(0)}<y_1<l_2^{(0)}=L$$
 $$l_1^{(1)}=L\left({1\over 1+\theta+L}\right)^2<x_2<L\left({1+L\over 1+\theta+L}\right)^2=l_2^{(1)}, \ \ \ 	l_1^{(1)}<y_2<l_2^{(1)}$$
 iterating we get
\begin{equation}\label{xn}
	l_1^{(n)}<x_{n+1}<l_2^{(n)}, \ \ \ 	l_1^{(n)}<y_{n+1}<l_2^{(n)},
\end{equation}	
where
\begin{equation}\label{lnp}
	\begin{array}{ll}
		l_1^{(n+1)}=L\cdot \left(\dfrac{1+l_1^{(n)}}{1+\theta+l_1^{(n)}+l_2^{(n)}}\right)^2\\[2mm]
		l_2^{(n+1)}=L\cdot \left(\dfrac{1+l_2^{(n)}}{1+\theta+l_1^{(n)}+l_2^{(n)}}\right)^2,
	\end{array}
\end{equation}
with $l_1^{(0)}=0 \ \ l_2^{(0)}= L.$

Again using that function $f(x,y)$ is monotone increasing with respect to $x$ and decreasing with respect to $y$
 one can see that
$$0=l_1^{(0)}<l_1^{(1)}<l_1^{(2)}<\dots <L,$$
$$L=l_2^{(0)}>l_2^{(1)}>l_2^{(2)}>\dots >0.$$

Both sequences are monotonic and bounded, so they have a limit. By taking the limit from both sides of (\ref{lnp}), we can see that the limit point is a fixed point of the operator $W$.
\end{proof}
\begin{pro}\label{limp} The following statements hold:
	\begin{itemize}
		\item[a)] The fixed point $\hat\ell=(\hat\ell_1, \hat\ell_2)\in {\rm Fix}(W)$ mentioned in Lemma \ref{ell} has property that
		$$\hat\ell_1=\min\{{\rm the\, first \, coordinates\, of \, all\, fixed \, points}\}.$$
			$$\hat\ell_2=\max\{{\rm the \, second \, coordinates\, of \, all\, fixed \, points}\}.$$
		\item[b)] If $|{\rm Fix}(W)|=1$, i.e. $W$ has a unique fixed point, $\hat \ell$, then trajectory of any initial point $v^{(0)}$ converges to this fixed point: $\omega(v^{(0)})=\{\hat \ell\}.$
		\item[c)] If $|{\rm Fix}(W)|\geq 2$ then $\hat\ell_1\ne \hat\ell_2$ and  for any initial point $v^{(0)}\in \mathbb R^2_+$ the set of limit points of its trajectory satisfies
		$$\omega(v^{(0)})\subset [\hat\ell_1, \hat\ell_2]^2.$$
	\end{itemize}
\end{pro}
\begin{proof}
a) For each  $(a,b)\in {\rm Fix}(W)$, taking $v^{(0)}=(a,b)$,  for this initial point we have $x_n=a$ and $y_n=b$ for all $n\geq 1$. For this constant sequence,  taking limit from both sides of (\ref{xn}) as $n\to \infty$ we get that
$$\hat\ell_1\leq a\leq \hat\ell_2, \ \ \hat\ell_1\leq b\leq \hat\ell_2.$$

b) By Lemma \ref{lfp} we know that in case of uniqueness of fixed point we have $\hat\ell_1=\hat\ell_2$, which for any initial point,  by (\ref{xn}) gives that
$$\lim_{n\to\infty}x_n=\lim_{n\to\infty}y_n=\hat\ell_1.$$

c) This is also consequence of inequalities (\ref{xn}), which says that independently, on existence of limits of $x_n$ and $y_n$ their limit points will be between $\hat\ell_1$ and $\hat\ell_2$.
\end{proof}

\subsubsection{Jacobian matrix}
By Lemma \ref{lfp} we know that
the number $n_-(W)$  of fixed points of operator $W$ on the set $M_-$ is as following
$$n_-(W)=\left\{\begin{array}{lll}
	0, \ \ \mbox{if} \ \ (\theta, L)\in A_{1,0}\\[2mm]
	1, \ \ \mbox{if} \ \ (\theta, L)\in \bigcup_{j=1}^3A_{j,2}\\[2mm]
	2, \ \ \mbox{if} \ \ (\theta, L)\in \bigcup_{j=1}^3A_{j,4}
	\end{array}
	\right.$$

If $(u, v)\in {\rm Fix}(W)$ (see Lemma \ref{lfp} for 4 possible coordinates $t_i$, $i=1,2,3,4$) is a fixed point then it satisfies
\begin{equation}\label{fop}
	\begin{array}{ll}
		u=L\cdot \left(\dfrac{1+u}{1+\theta+u+v}\right)^2\\[2mm]
		v=L\cdot\left(\dfrac{1+v}{1+\theta+u+v}\right)^2,
	\end{array}
\end{equation} using this system we simplify Jacobian of the operator $W$ at the fixed point $(u, v)$ as
$$J(u,v)=\left(\begin{matrix} \frac{2u(\theta+v)}{(1+\theta+u+v)(1+u)} & -\frac{2u}{1+\theta+u+v} \\ -\frac{2u}{1+\theta+u+v} & \frac{2v(\theta+u)}{(1+\theta+u+v)(1+v)}\end{matrix}\right).$$

%

Since we have explicit formula for each fixed points of operator $W$, using formula of $J(u,v)$ one can find all  eigenvalues of this matrix. But the they will be in a bulky form. 
Here we give an analysis in case when $(u,v)=p_2$ (see Lemma \ref{lfp}).
Denote
$$t=\frac{1}{2}\left(L-2-2\theta+\sqrt{L(L+4-4\theta)}\right).$$
Then for eigenvalues of $J(p_2)$ we have
$$\mu_1=\frac{\theta+1}{\theta+1+t}+\frac{1}{(\theta+1+t)}\sqrt{(\theta-1)^2+4-\frac{4(\theta-1)^2}{2+t}},$$
$$\mu_2=\frac{\theta+1}{\theta+1+t}-\frac{1}{(\theta+1+t)}\sqrt{(\theta-1)^2+4-\frac{4(\theta-1)^2}{2+t}}.$$
We note that  $t\geq2$.
Simple computations show that

a) If $t=2$ then $\mu_1=1$ and  $|\mu_2|=|\frac{\theta-1}{\theta+3}|<1.$

b) If $t>2$ then

$$0<\mu_1\ \ \mbox{is} \ \ \left\{\begin{array}{ll}
	<1, \ \ \mbox{if} \ \ \theta\leq 5 \ \ \mbox{or} \ \ \theta>5, \, L>4(\theta-1)\\[2mm]
	=1, \ \ \mbox{if} \ \ \theta>5,  L=4(\theta-1)
	\end{array}\right.
	 $$
and	
	$$ |\mu_2|<1, \ \  \forall L\geq 4(\theta-1).$$
%
%
%
%
From well-known theorems about attractive fixed points of non-linear dynamical systems it follows the following
\begin{pro} The following assertions hold
	\begin{itemize}
		\item If $L\geq 4(\theta-1)$ then there is a neighborhood $U_2\subset M_-$ of $p_2$ such that for any $(x^{(0)}, y^{(0)})\in U_2$ we have
	$$\lim_{n\to +\infty}W^n(x^{(0)}, y^{(0)})=p_2.$$

	\item If $L<4(\theta-1)$ then
		$$\lim_{n\to +\infty}W^n(x^{(0)}, y^{(0)})\in M_0.$$
				\end{itemize}
	\end{pro}

\subsubsection{A sufficient condition of convergence to $M_0$}

On the set  $M_-$ we introduce new variables:
$$s=x+y, \ \ t=y-x.$$
Then the operator $W$ can be rewritten as
\begin{equation}\label{opp}
\hat W:\left\{	\begin{array}{ll}
		s'=L \dfrac{(2+s)^2+t^2}{2(1+\theta+s)^2}\\[2mm]
		t'=L \dfrac{2+s}{(1+\theta+s)^2}t,
	\end{array}\right.
\end{equation}
Define
$$\psi(s)=\dfrac{2+s}{(1+\theta+s)^2}.$$
It is easy to see that
$$\max_{s>0}\psi(s)=\Lambda(\theta):=\left\{\begin{array}{ll}
	\dfrac{1}{4(\theta-1)}, \ \ \mbox{if} \ \ \theta>3\\[2mm]
	\dfrac{2}{(\theta+1)^2}, \ \ \mbox{if} \ \ \theta\leq 3
\end{array}	\right.$$
\begin{pro} If parameters $\theta>0$ and $L>0$ such that
	$L\Lambda(\theta)<1$ then for any $z_0=(x_0, y_0)\in \mathbb R_+^2$ the limit set of trajectory $z_{n+1}=W^n(z_n),$ $n\geq 0$ is a subset of $M_0$.
\end{pro}
\begin{proof} Let $(s_n,t_n)=\hat W^n(x_0+y_0, y_0-x_0)$. We have to show that $\lim_{n\to \infty}t_n=0$. Consider $z_0\in M_-$ (the case $z_0\in M_0$ is trivial, and the case $z_0\in M_+$ is similar to $M_-$).
	From (\ref{opp}) we get for $n\geq 1$ that
	$$0<t_{n+1}=L \dfrac{2+s_n}{(1+\theta+s_n)^2}t_n=L\psi(s_n)t_n\leq (L\Lambda(\theta))^nt_0\to 0, \ \ n\to \infty.$$
\end{proof}
\subsubsection{Limit points}
In this subsection for each fixed point we study its set of  attraction and set of repulsion. To do this we use the property that $W$ is monotone with respect to a partial order on $\mathbb R^2$ (see methods of \cite{BB} and the references therein).

Recall some definitions (see \cite{BB}). Let $\preceq$ be a partial order on $\mathbb{R}^n$ with non negative cone $P$. For $x, y \in \mathbb{R}^n$ the order interval $[[x, y]]$ is the set of all $z$ such that $x \preceq z \preceq y$. One says $x \prec y$ if $x \preceq y$ and $x \neq y$.
A map $T$ on a subset of $\mathbb{R}^n$ is order preserving  (monotone) if $T(x) \preceq T(y)$ whenever $x \preceq y$, strictly order preserving if $T(x) \prec T(y)$ whenever $x \prec y$.

Let $T: \mathbb R \rightarrow \mathbb R$ be a map with a fixed point $\bar{x}$ and let $R^{\prime}$ be an invariant subset of $R$ that contains $\bar{x}$. We say that $\bar{x}$ is stable (asymptotically stable) relative to $R^{\prime}$ if $\bar{x}$ is a stable (asymptotically stable) fixed point of the restriction of $T$ to $R^{\prime}$.

We will use the next theorem \cite{KM}, which is stated for order preserving maps on $\mathbb{R}^n$ (see also \cite{BB}).

\begin{thm}\label{tK} For a nonempty set $R \subset \mathbb{R}^n$ and $\preceq a$ partial order on $\mathbb{R}^n$, let $T: R \rightarrow R$ be an order preserving map, and let $a, b \in R$ be such that $a \prec b$ and $[[ a, b]] \subset R$. If $a \preceq T(a)$ and $T(b) \preceq b$, then
\begin{itemize}
	\item[(i)] 	$[[a, b]]$ is an invariant set;
\item[(ii)] There exists a fixed point of $T$ in $[[a, b]]$;
\item[(iii)] If $T$ is strongly order preserving, then there exists a fixed point in $[[a, b]]$ which is stable relative to $[[a, b]]$;
\item[(iv)] If there is only one fixed point in $[[a, b]]$, then it is a global attractor, asymptotically stable relative to $[[a, b]]$.
\end{itemize}
\end{thm}
 The following result is a direct consequence of
the Trichotomy Theorem of \cite{DH} (see also \cite{BB}, \cite{KM}), and is helpful for determining the
basins of attraction of the fixed points.

\begin{pro}\label{bas} If the non-negative cone of a partial ordering $\preceq$ is a generalized quadrant in $\mathbb{R}^n$, and if $T$ has no fixed points in $[[u_1, u_2]]$ other than $u_1$ and $u_2$, then the interior of $[[u_1, u_2]]$ is either a subset of the basin of attraction of $u_1$ or a subset of the basin of attraction of $u_2$.
\end{pro}

Consider the North-East ordering (NE) on $\mathbb R^2$ for which the positive cone is the first quadrant, i.e., this partial ordering is defined by $\left(x_1, y_1\right) \preceq_{ne}\left(x_2, y_2\right)$ if $x_1 \leq x_2$ and $y_1 \leq y_2$. The South-East (SE) ordering defined as $\left(x_1, y_1\right) \preceq_{se}\left(x_2, y_2\right)$ if $x_1 \leq x_2$ and $y_1 \geq y_2$.

A map $T$ on a nonempty set $\mathcal{R} \subset \mathbb{R}^2$ which is monotone with respect to the NE ordering is called cooperative and a map monotone with respect to the SE ordering is called competitive.

\begin{lemma}\label{se}
	The operator $W$ defined in (\ref{op}) is  monotone with respect to the SE ordering, i.e., is a competitive map.
\end{lemma}
\begin{proof}
	By (\ref{op})	we have (for $L_1=L_2=L$)
	$$x_1'-x_2'=L\left[\left({1+x_1\over 1+\theta+x_1+y_1}\right)+
	\left({1+x_2\over 1+\theta+x_2+y_2}\right)\right]\times$$ $$\left[\left({1+x_1\over 1+\theta+x_1+y_1}\right)-
	\left({1+x_2\over 1+\theta+x_2+y_2}\right)\right]$$
	$$=A\cdot [(\theta+y_2)(x_1-x_2)+(1+x_2)(y_2-y_1)],$$
	where $A=A(x_1,x_2,y_1,y_2,\theta)>0$.  Such an equality also can be written for $y_1'-y_2'$. Thus
	\begin{equation}\label{XY}\begin{array}{ll}
			x_1'-x_2'=A\cdot [(\theta+y_2)(x_1-x_2)+(1+x_2)(y_2-y_1)],\\[2mm]
			y_1'-y_2'=B\cdot [(1+y_2)(x_2-x_1)+(\theta+x_2)(y_1-y_2)],
		\end{array}
	\end{equation}
	where $B=B(x_1,x_2,y_1,y_2,\theta)>0$.
	
	Take $(x_1, y_1)\preceq_{se} (x_2, y_2)\in \mathbb R^2$ then $x_1\leq x_2$ and $y_1\geq y_2$. By the last equalities from  (\ref{XY}) one can see that $x_1'\leq x_2'$ and $y_1'\geq y_2'$, i.e.,
	$$(x_1', y_1')=W(x_1, y_1)\preceq_{se} (x_2', y_2')=W(x_2, y_2).$$
	This completes the proof.	
\end{proof}
	The operator (\ref{op}) does not exhibit monotonicity with respect to the NE ordering. The following proposition gives relations between SE and NE.
\begin{pro}\label{hamma}
	Let $(a,b)\in {\rm Fix}(W)$ and $(x,y)\in \mathbb R^2_+$.
	\begin{itemize}
		\item[1.] If $(x,y)\preceq_{se} (a,b)$ then $W(x,y)\preceq_{se} (a,b)$.
		\item[2.] If $(a,b)\preceq_{se} (x,y)$ then $(a,b)\preceq_{se} W(x,y)$.
		\item[3.] Assume $\theta\geq 1$.
		\begin{itemize}
			\item[3a.]	If $$(x,y)\in \left\{(x,y)\in \mathbb R^2_+: \, (x,y)\preceq_{ne} (a,b), \, b-{\theta+b\over 1+a}(a-x)\leq y \leq b-{1+b\over \theta+a}(a-x)\right\}$$ then
			$W(x,y)\preceq_{ne} (a,b)$.
			\item[3b.] If
			$$(x,y)\in \left\{(x,y)\in \mathbb R^2_+: \, (x,y)\preceq_{ne} (a,b), \, y\leq b-{\theta+b\over 1+a}(a-x)\right\}$$ then
			$(a,b)\preceq_{se} W(x,y)$.
			\item[3c.]  If
			$$(x,y)\in \left\{(x,y)\in \mathbb R^2_+: \, (x,y)\preceq_{ne} (a,b), \, b-{1+b\over \theta+a}(a-x)\leq y\right\}$$ then
			$W(x,y)\preceq_{se} (a,b)$.
			\item[3a'.]	If $$(x,y)\in \left\{(x,y)\in \mathbb R^2_+: \, (a,b)\preceq_{ne} (x,y), \, b+{1+b\over \theta+a}(x-a)\leq y \leq b+{\theta+b\over 1+a}(x-a)\right\}$$ then
			$(a,b) \preceq_{ne} W(x,y)$.
			\item[3b'.] If
			$$(x,y)\in \left\{(x,y)\in \mathbb R^2_+: \, (a,b)\preceq_{ne} (x,y), \, y\leq b+{1+b\over \theta+a}(x-a)\right\}$$ then
			$(a,b)\preceq_{se} W(x,y)$.
			\item[3c'.]  If
			$$(x,y)\in \left\{(x,y)\in \mathbb R^2_+: \, (a,b)\preceq_{ne} (x,y), \, b+{\theta+b\over 1+a}(x-a)\leq y\right\}$$ then
			$W(x,y)\preceq_{se} (a,b)$.
		\end{itemize}
		\item[4.] Assume $\theta< 1$.
		\begin{itemize}
			\item[4a.]	If $$(x,y)\in \left\{(x,y)\in \mathbb R^2_+: \, (x,y)\preceq_{ne} (a,b), \, b-{1+b\over \theta+a}(a-x)\leq y \leq b-{\theta+b\over 1+a}(a-x)\right\}$$ then
			$(a,b)\preceq_{ne} W(x,y)$.
			\item[4b.] If
			$$(x,y)\in \left\{(x,y)\in \mathbb R^2_+: \, (x,y)\preceq_{ne} (a,b), \, y\leq b-{1+b\over \theta+a}(a-x)\right\}$$ then
			$(a,b)\preceq_{se} W(x,y)$.
			\item[4c.]  If
			$$(x,y)\in \left\{(x,y)\in \mathbb R^2_+: \, (x,y)\preceq_{ne} (a,b), \, b-{\theta+b\over 1+a}(a-x)\leq y\right\}$$ then
			$W(x,y)\preceq_{se} (a,b)$.
			\item[4a'.]	If $$(x,y)\in \left\{(x,y)\in \mathbb R^2_+: \, (a, b)\preceq_{ne} (x, y), \, b+{\theta+b\over 1+a}(x-a)\leq y \leq b+{1+b\over \theta+a}(x-a)\right\}$$ then
			$W(x,y)\preceq_{ne} (a, b)$.
			\item[4b'.] If
			$$(x,y)\in \left\{(x,y)\in \mathbb R^2_+: \, (a, b)\preceq_{ne} (x, y), \, y\leq b+{\theta+b\over 1+a}(x-a)\right\}$$ then
			$(a,b)\preceq_{se} W(x,y)$.
			\item[4c'.]  If
			$$(x,y)\in \left\{(x,y)\in \mathbb R^2_+: \, (a, b)\preceq_{ne} (x, y), \, b+{1+b\over \theta+a}(x-a)\leq y\right\}$$ then
			$W(x,y)\preceq_{se} (a,b)$.
		\end{itemize}
	\end{itemize}
\end{pro}
\begin{proof} Parts 1-2 follow from Lemma \ref{se}.
	For $(x,y)\in \mathbb R_+^2$ we write $(x', y')=W(x,y)$. Then taking $(x_1, y_1)=(x, y)$, $(x_2, y_2)=(a, b)$, since $(a,b)$ is a fixed point, from (\ref{XY}) we get
	\begin{equation}\label{AB}\begin{array}{ll}
			x'-a=A\cdot [(\theta+b)(x-a)+(1+a)(b-y)],\\[2mm]
			y'-b=-B\cdot [(1+b)(x-a)+(\theta+a)(b-y)],
		\end{array}
	\end{equation}
	where $A>0$, $B>0$. For each item mentioned in proposition  by (\ref{AB}) one has to solve a system of linear inequalities.
\end{proof}

For $x>0$ define function
$$\psi(x)=\sqrt{L}\left(\frac{1}{\sqrt{x}}+\sqrt{x}\right)-\left(1+\theta+x\right).$$
Consider the following partitions of $\mathbb R^2_+$:
$$\mathbb R^2_+=A_{se}^<\cup A_{se}^>\cup A_{ne}^<\cup A_{ne}^>, $$
where
$$A_{se}^<=\{(x,y)\in \mathbb R_+^2: x\geq \psi(y), \, y\leq \psi(x)\},$$
$$A_{se}^>=\{(x,y)\in \mathbb R_+^2: x\leq \psi(y), \, y\geq \psi(x)\},$$
$$A_{ne}^<=\{(x,y)\in \mathbb R_+^2: x\leq \psi(y), \, y\leq \psi(x)\},$$
$$A_{ne}^>=\{(x,y)\in \mathbb R_+^2: x\geq \psi(y), \, y\geq \psi(x)\}.$$

\begin{pro}\label{psi} We have
	
	$$A_{se}^<=\{(x,y)\in \mathbb R^2_+: (x,y)\preceq_{se} W(x,y)\}.$$
	$$A_{se}^>=\{(x,y)\in \mathbb R^2_+: W(x,y)\preceq_{se} (x,y)\}.$$
	$${\rm Fix}(W)=\{(x,y)\in \mathbb R^2_+: x=\psi(y), \, y=\psi(x)\}.$$
	$$A_{ne}^<=\{(x,y)\in \mathbb R^2_+: (x,y)\preceq_{ne} W(x,y)\}.$$
	$$A_{ne}^>=\{(x,y)\in \mathbb R^2_+: W(x,y)\preceq_{ne} (x,y)\}.$$
	
\end{pro}
\begin{proof} Let us prove the first equality. Left hand side is
$$\begin{array}{ll}
	x\leq x'=L\left(1+x\over 1+\theta+x+y\right)^2,\\[3mm]
y\geq y'=L\left(1+y\over 1+\theta+x+y\right)^2.
\end{array}
$$
Rewrite this system as:
$$\begin{array}{ll}\sqrt{L}\left(1+x\over 1+\theta+x+y\right)-\sqrt{x}\geq0,\\[3mm]
\sqrt{L}\left(1+y\over 1+\theta+x+y\right)-\sqrt{y}\leq0.
\end{array}
$$
Consequently,
$$\begin{array}{ll}\sqrt{L}\left(\frac{1}{\sqrt{x}}+\sqrt{x}\right)-\left(1+\theta+x\right)\geq y,\\[3mm]
\sqrt{L}\left(\frac{1}{\sqrt{y}}+\sqrt{y}\right)-\left(1+\theta+y\right)\leq x.\end{array}$$
Hence
$$y\leq\psi(x), \ \ x\geq\psi(y).$$
This completes the proof (see Fig. \ref{fa1}).
\end{proof}
\begin{figure}
\begin{center}
\begin {tikzpicture} [scale=0.8]
\draw[->, thick] (0,0) -- (0,10);
\draw[->, thick] (0,0) -- (11,0);
\draw[] (1,-0.1) -- (1,0.1);
\draw[] (2,-0.1) -- (2,0.1);
\draw[] (3,-0.1) -- (3,0.1);
\draw[] (4,-0.1) -- (4,0.1);
\draw[] (5,-0.1) -- (5,0.1);
\draw[] (6,-0.1) -- (6,0.1);
\draw[] (7,-0.1) -- (7,0.1);
\draw[] (8,-0.1) -- (8,0.1);
\draw[] (9,-0.1) -- (9,0.1);
\draw[] (10,-0.1) -- (10,0.1);
\draw[] (-0.1,1) -- (0.1,1);
\draw[] (-0.1,2) -- (0.1,2);
\draw[] (-0.1,3) -- (0.1,3);
\draw[] (-0.1,4) -- (0.1,4);
\draw[] (-0.1,5) -- (0.1,5);
\draw[] (-0.1,6) -- (0.1,6);
\draw[] (-0.1,7) -- (0.1,7);
\draw[] (-0.1,8) -- (0.1,8);
\draw[] (-0.1,9) -- (0.1,9);
\draw[] (0,0)-- (9,9);
\draw[thick, red] (0.5,10).. controls (0.52,0.3) .. (1, 0.2) .. controls (2.4,4) and (3,3) .. (9,0);
\draw[thick, green] (11 ,0.5).. controls (0.3,0.52) .. (0.2,1) .. controls (4,2.4) and (3,3) .. (0,9);
\node[below] at (1,0){1};
\node[below] at (2,0){2};
\node[below] at (3,0){3};
\node[below] at (4,0){4};
\node[below] at (5,0){5};
\node[below] at (6,0){6};
\node[below] at (7,0){7};
\node[below] at (8,0){8};
\node[below] at (9,0){9};
\node[below] at (10,0){10};
\node[above] at (11,0){$\lambda_1$};
\node[left] at (0,1){1};
\node[left] at (0,2){2};
\node[left] at (0,3){3};
\node[left] at (0,4){4};
\node[left] at (0,5){5};
\node[left] at (0,6){6};
\node[left] at (0,7){7};
\node[left] at (0,8){8};
\node[left] at (0,9){9};
\node[right] at (0,10){$L$};
\node[right] at (0.6,8){$p_4$};
\node[right] at (8,0.7){$p_3$};
\node[left] at (0.75,0.4){$p_1^*$};
\node[left] at (0.75,1.3){$p_2$};
\node[left] at (1.9,0.8){$p_1$};
\node[left] at (1.7,1.8){$p_2^*$};
\node[left] at (2.65,2.8){$p_3^*$};
\end{tikzpicture}
\caption{Graphs of $x=\psi(y)$ (green) and $y=\psi(x)$ (red) in case of having 7 fixed points. The sets mentioned in Proposition \ref{psi} can be seen between red and green curves.}\label{fa1}
\end{center}
\end{figure}
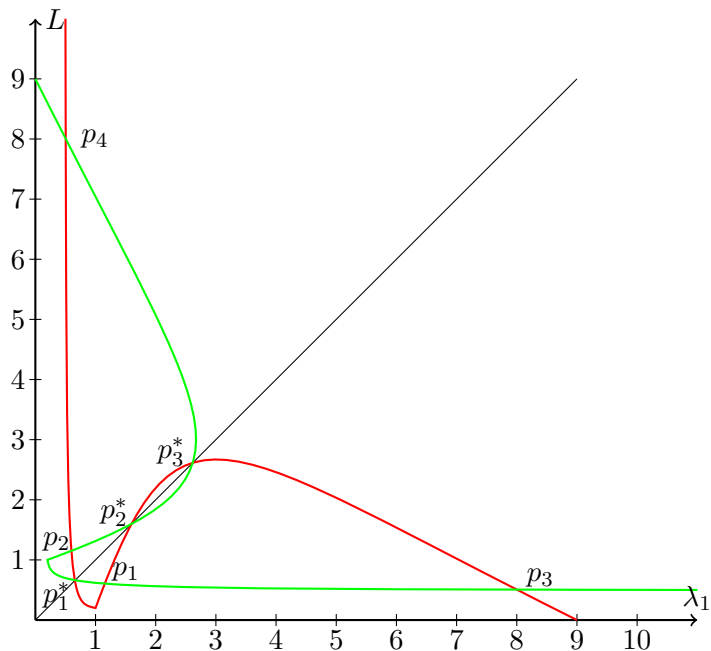
For a fixed point $p\in {\rm Fix}(W)$, we denote
$$B^<(p)=\{(x,y)\in M_-: (x,y)\prec_{se} W(x,y)\preceq_{se} p\},$$
$$B^>(p)=\{(x,y)\in M_-: p\preceq_{se} W(x,y)\prec_{se} (x,y)\},$$
$$R^<(p)=\{(x,y)\in M_-: W(x,y)\prec_{se} (x,y)\preceq_{se} p\},$$
$$R^>(p)=\{(x,y)\in M_-: p\preceq_{se} (x,y)\prec_{se} W(x,y)\}.$$

\begin{lemma} For any $p\in {\rm Fix}(W)$ the sets  $B^<(p)$, $B^>(p)$, $R^<(p)$, $R^>(p)$ are invariant with respect to $W$.
\end{lemma}
\begin{proof}
	Follows from Lemma \ref{se} and part 1, part 2 of Proposition \ref{hamma}.
\end{proof}
Recall notations given in and before Lemma \ref{lfp}, then by Proposition \ref{limp} and above mentioned lemmas we get
\begin{thm}\label{ta} For initial point $v^{(0)}\in \mathbb R^2_+\setminus M_+$, the following statements hold
$$\lim_{n\to\infty}W^n(v^{(0)})=\left\{\begin{array}{lllllll}
	p_1^*, \ \ \mbox{if} \ \ (\theta, L)\in A_{1,0};\\[2mm]
	p_2, \ \ \mbox{if} \ \ (\theta, L)\in A_{1,2}, \forall v^{(0)}\in M_-;\\[2mm]
p_2, \ \ \mbox{if} \ \ (\theta, L)\in A_{1,4},
\forall v^{(0)}\in (B^<(p_2)\cup B^>(p_2))\bigcup(R^<(p_4)\cup R^>(p_4));\\[2mm]
p_4, \ \ \mbox{if} \ \ (\theta, L)\in A_{1,4},
\forall v^{(0)}\in (B^<(p_4)\cup B^>(p_4))\bigcup(R^<(p_2)\cup R^>(p_2));\\[2mm]
p, \ \ \mbox{if} \ \ (\theta, L)\in \bigcup_{{i=2,3\atop j=2,4}} A_{i,j}, p\in {\rm Fix}(W),
\forall v^{(0)}\in (B^<(p)\cup B^>(p)).
\end{array}
	\right.
	$$
	\end{thm}
\begin{rk} Let us note that
\begin{itemize}
	\item[1.] Theorem \ref{ta} is true for $M_+$ too, one has only replace $p_2, p_4$ by $p_1, p_3$ respectively.
	\item[2.] The union of subsets mentioned for $v^{(0)}$ does not cover $M_-$, so for the remaining possibilities of the initial point, we do not know their set of limit points. However, Proposition \ref{limp} states that, regardless of the convergence behavior, the set of limit points of any trajectory is a subset of $[\hat\ell_1, \hat\ell_2]^2$.
	\item[3.] If type of each fixed point is known then the last line of Theorem \ref{ta} can be given in more detail, using sets $R^<(p)$ and $R^>(p)$ too.
	\end{itemize}
\end{rk}
For an initial point $x^{(0)}=(x_1^{(0)}, x_2^{(0)}, \dots)\in  l^{1}_{+}$, introduce
\begin{equation}\label{vvo}v^{(0)}_1=\sum_{j=1}^{\infty}x^{(0)}_{2j-1}, \ \ v^{(0)}_2=\sum_{j=1}^{\infty}x^{(0)}_{2j}.
\end{equation}

By Lemma \ref{Le} and Remark \ref{ob} from Theorem \ref{ta} we obtain that
	\begin{thm}
If	initial point $x^{(0)}=(x_1^{(0)}, x_2^{(0)}, \dots)\in  l^{1}_{+}$ is such that
	$$\lim_{m\to \infty}W^m(v_1^{(0)}, v_2^{(0)})=(a,b)=p\in {\rm Fix}(W)$$
	for corresponding values (\ref{vvo}) then
	$$\lim_{m\to \infty}F^m(x^{(0)})={1\over L}\left(a\lambda_1, b\lambda_2, a\lambda_3, b\lambda_4, \dots\right).
	$$
\end{thm}

\section*{Data availability statements}
The datasets generated during and/or analysed during the current study are available from the corresponding author (U.A.Rozikov) on reasonable request.

\section*{Conflicts of interest} The authors declare no conflicts of interest.

\section*{Acknowledgements}

U.~Rozikov thanks the Weierstrass Institute for Applied Analysis and Stochastics, Berlin, Germany for support of his visit.  His  work was partially supported through a grant from the IMU--CDC and the fundamental project (grant no.~F--FA--2021--425) of The Ministry of Innovative Development of the Republic of Uzbekistan.

\end{document}